\documentclass{article}
\usepackage{amsmath, amsthm}
\usepackage{amsfonts}
\usepackage{amssymb}
\usepackage{amsthm}
\usepackage{amsmath}
\usepackage{latexsym}
\usepackage[all]{xy}
\usepackage{mathrsfs}
\usepackage{graphicx}

\newtheoremstyle{theorem}
  {10pt}		  
  {10pt}  
  {\sl}  
  {\parindent}     
  {\bf}  
  {. }    
  { }    
  {}     
\theoremstyle{theorem}
\newtheorem{theorem}{Theorem}
\newtheorem{corollary}[theorem]{Corollary}
\newtheorem{lemma}[theorem]{Lemma}

\newtheoremstyle{defi}
  {10pt}		  
  {10pt}  
  {\rm}  
  {\parindent}     
  {\bf}  
  {. }    
  { }    
  {}     
\theoremstyle{defi}

\title{The inverse problem for the lattice points}
\author{\v Zeljka Ljuji\' c, Camilo Sanabria}

\begin{document}

\maketitle

\begin{abstract}
Fix an positive integer $n$. Let $K\subseteq\mathbb{R}^n$ be a compact set such that $K+\mathbb{Z}^n=\mathbb{R}^n$. We prove, via Algebraic Topology, that the integer points of the difference set of $K$, $(K-K)\cap\mathbb{Z}^n$, is not contained on the coordinate axes,
$\mathbb{Z}\times\{0\}\times\ldots\times\{0\}\cup\{0\}\times\mathbb{Z}\times\ldots\times\{0\}\cup\ldots\cup\{0\}\times\{0\}\times\ldots\times\mathbb{Z}$. This result gives a negative answer to a question posed by P. Hegarty and M. Nathanson on relatively prime lattice points.
\end{abstract}

Let $n\ge 1$. We say that the set $A\subset\mathbb{Z}^n$ is relatively prime if the elements of $A$ generates $\mathbb{Z}^n$. An $\mathcal{N}$-set is a compact set $K\subset\mathbb{R}^n$ such that for every $x\in\mathbb{R}^n$ there exists $y\in K$ with $x\equiv y$ (mod $\mathbb{Z}^n$). M. Nathanson proved in \cite{1} the following theorem
\begin{theorem} Let $A$ be a set of positive integers. The set is relatively prime if and only if there exist an $\mathcal{N}$-set K in $\mathbb{R}$ such that $A=(K-K)\cap\mathbb{N}$.
\end{theorem}

The proof of the sufficiency condition for a set $A$ to be relatively prime follows from a more general result from geometric group theory that holds for any dimension:
\begin{theorem} If $K$ is an $\mathcal{N}$-set in $\mathbb{R}^n$, then $A=(K-K)\cap\mathbb{Z}^n$ is a finite set of relatively prime lattice points.
\end{theorem}

The other direction, the fact that every finite set of relatively prime positive integers can be realized as the difference of an $\mathcal{N}$-set was proved by giving an explicit construction of such an $\mathcal{N}$-set K for a prescribed set $A$. This problem can be seen as an inverse problem. 

In the same paper, the following natural question was asked:

\emph{Does every finite symmetric set of relatively prime lattice points that contains $0$ is of the form $(K-K)\cap\mathbb{Z}^n$ for some $\mathcal{N}$-set $K$?} 

The former theorem answers this question in the case $n=1$. The answer was not known for the higher dimensions. 

As an attempt to attack the case $n=2$, P. Hegarty raised the following question: Does there exist an $\mathcal{N}$-set $K$ such that $(K-K)\cap\mathbb{Z}^2\subset(\mathbb{Z}\times\{0\})\cup(\{0\}\times\mathbb{Z})$? This problem was solved independently by Z. Ljujic and C. Sanabria in \cite{4} and by L. Borisov and R. Jin in \cite{2}. They proved that the answer to this question is ``no'', by proving that the set $A=\{(-1,0), (0,-1), (0,0), (1,0), (0,1)\}$, although a finite, symmetric set of relatively prime lattice points containing $0$, is not of the form $(K-K)\cap\mathbb{Z}^2$ for any $\mathcal{N}$-set $K$. Whence, in the case $n=2$, the answer to the inverse problem is negative.

Here, we are giving an argument that will confirm the negative answer to the same question but for any $n>1$. We will start the proof by using the observation of L. Borisov and R. Jin that says that instead of considering any compact set $K$, it is enough to consider a set $B=\bigcup_{i=0}^{m-1}\bigcup_{j=0}^{m-1}B_{i,j}+u_{i,j}$, where $B_{i,j}=[\frac{i}{m},\frac{i+1}{m}]\times[\frac{j}{m},\frac{j+1}{m}]$ and $u_{i,j}\in\mathbb{Z}^2$, for some $m$. This result was generalized for higher dimensions by M. Nathanson, in \cite{3}, 
\begin{theorem} Let $K$ be a compact set of $\mathbb{R}^n$. For $J=(j_1,\dots,j_n)\in\mathbb{Z}^n$, let
\[C_{N,J}=[\frac{j_1}{2^N},\frac{j_1+1}{2^N}]\times\cdots\times[\frac{j_n}{2^N},\frac{j_n+1}{2^N}].
\]
\noindent There exists an integer $N_0$ such that for every integer $N\ge N_0$ there is a finite subset $\mathcal{J}_n$ of $\mathbb{Z}^n$ such that the set 
\[K_N=\bigcup_{J\in\mathcal{J}_n}C_{N,J}
\]
satisfies $K\subset K_N$ and
\[(K-K)\cap\mathbb{Z}^n=(K_N-K_N)\cap\mathbb{Z}^n.
\]
\end{theorem}

We proceed to the proof of the negative answer to the question.

\begin{theorem}\label{thm}  It does not exist a compact set $K$ s.t. $\mathbb{R}^n=K+\mathbb{Z}^n$ and $(K-K)\cap\mathbb{Z}^n\subseteq(\mathbb{Z}\times\{0\}\times\ldots\times\{0\}\cup\{0\}\times\mathbb{Z}\times\ldots\times\{0\}\cup\ldots\cup\{0\}\times\{0\}\times\ldots\times\mathbb{Z})$.
\end{theorem}
\begin{proof} Let us assume that such set $K$ exists. Then it exists $m\in\mathbb{Z}_{>0}$ and set $B=\bigcup_{j_1=0}^{m-1}\ldots\bigcup_{j_n=0}^{m-1}B_{(j_1,\ldots,j_n)}+u_{(j_1,\ldots,j_n)}$, where $B_{(j_1,\ldots,j_n)}=[\frac{j_1}{m},\frac{j_1+1}{m}]\times\ldots,\times[\frac{j_n}{m},\frac{j+1}{m}]$ and $u_{(j_1,\ldots,j_n)}\in\mathbb{Z}^n$, such that $(B-B)\cap\mathbb{Z}^n\subseteq(\mathbb{Z}\times\{0\}\times\ldots\times\{0\}\cup\ldots\cup\{0\}\times\{0\}\times\ldots\times\mathbb{Z})$. Translating $K$ by $-u_{(0,\ldots,0)}$, we may assume $u_{(0,\ldots,0)}=(0,\ldots,0)$.

Let us consider the unit square subdivided into $m^n$ squares $B_{(j_1,\ldots,j_n)}$, where $0\leqslant j_1,\ldots,j_n \leqslant m-1$. We label the vertices $(\frac{j_1}{m},\ldots,\frac{j_n}{m})$, where $0\leqslant j_1,\ldots,j_n \leqslant m$, with the value $v_{(j_1,\ldots,j_n)}$ in the following way
\begin{equation*} 
v_{(j_1,\ldots,j_n)} = \left\{ 
\begin{array}{rl} 
u_{(j_1,\ldots,j_n)} & \text{if } 0\leqslant j_1,\ldots,j_n \leqslant m-1\\ 
u_{(i_1,\ldots,i_n)} -\sum_{l=1,j_l=m}^{l=n} e_l & \text{where } i_l=j_l \text{ if } 0\leqslant j_l \leqslant m-1\\
                                                                     & \text{and } i_l=0 \text{ if } j_l=m;
\end{array} \right. 
\end{equation*} 
\noindent and where $e_l$ is the $l$th vector in the standard basis, i.e. the vector with a $1$ on the $l$th entry and zeroes everywhere else.

We direct the edges, i.e. the edges of the $B_{(j_1,\ldots,j_n)}$'s, in the standard positive direction, i.e. from $u_{(j_1,\ldots,j_n)}$ to $u_{(j_1,\ldots,j_n)}+\frac{e_l}{m}$, and we label them with the value of the differences: value at the ending vertex minus value at the initial vertex. Note, that the unit square subdivided in this fashion, and with the standard orientation, can be seen as the torus $T$ with a given $\square$-complex structure. If we denote the labeling of the edges by $\psi$, then
\begin{equation*}
\psi([u_{(j_1,\ldots,j_n)}, u_{(j_1,\ldots,j_n)+e_l}])=v_{(j_1,\ldots,j_n)+e_l}-v_{(j_1,\ldots,j_n)},
\end{equation*}
\noindent for $0\leqslant j_l \leqslant n-1, 0\leqslant j_1,\ldots,\hat{j_l},\ldots,j_n \leqslant n$. Note that 
\begin{equation*}
\psi([u_{(0,j_2,\ldots,j_n)}, u_{(0,j_2,\ldots,j_n)+e_l}])=\psi(u_{(m,j_2,\ldots,j_n)}, u_{(m,j_2,\ldots,j_n)+e_l}]),
\end{equation*}
for $0\leqslant j_2,\ldots,j_n \leqslant m-1$ and $e_l\ne e_1$;
\begin{equation*}
\psi([u_{(j_1,0,\ldots,j_n)}, u_{(j_1,0,\ldots,j_n)+e_l}])=\psi(u_{(j_1,m,\ldots,j_n)}, u_{(j_1,m,\ldots,j_n)+e_l}]),
\end{equation*}
for $0\leqslant j_1,j_3\ldots,j_n \leqslant m-1$ and $e_l\ne e_2$; $\ldots$ and
\begin{equation*}
\psi([u_{(j_1,\ldots,0)}, u_{(j_1,\ldots,0)+e_l}])=\psi([u_{(j_1,\ldots,m)}, u_{(j_1,\ldots,m)}]),
\end{equation*}
for $0\leqslant j_1,j_2\ldots,j_{n-1} \leqslant m-1$ and $e_l\ne e_n$, so $\psi$ is a well-defined function from the edges of $T$ to the abelian group $\mathbb{Z}^n$. Moreover,
\begin{eqnarray*}
\psi([u_{(j_1,\ldots,j_n)}, u_{(j_1,\ldots,j_n)+e_l}]) & + & \psi([u_{(j_1,\ldots,j_n)+e_l}, u_{(j_1,\ldots,j_n)+e_l+e_k}])\\
-\psi([u_{(j_1,\ldots,j_n)+e_k}, u_{(j_1,\ldots,j_n)+e_l+e_k}]) & - & \psi([u_{(j_1,\ldots,j_n)}, u_{(j_1,\ldots,j_n)+e_k}])=0
\end{eqnarray*}
for $0\leqslant i, j\leqslant n-1$, so we can see $\psi$ as one representative of an element of the cohomology group $H^1(T;\mathbb{Z}^n)$. There is a natural map $h:H^1(T;\mathbb{Z}^n)\rightarrow\textrm{Hom}(H_1(T),\mathbb{Z}^n)$ that sends $\psi$ to $\overline{\psi_0}:H_1(T)\rightarrow\mathbb{Z}^n$ where $\overline{\psi_0}([[0,e_l]])=-e_l$, for $1\leqslant l \leqslant n$. Here, we were using that $H_1(T)=\mathbb{Z}^n$ with basis the homology classes $[[0,e_1]],\ldots,[[0,e_n]]$. Hence, $\overline{\psi_0}$ is an isomorphism. This means that
\[
(*) \left. \begin{array}{l}\textrm{we can read the homotopy type of a closed curve from the sum} \\ \textrm{of the values associated by } \psi \textrm{ to the edges forming the curve.} \end{array} \right.
\]

All the values associated to the edges are lying in the set $B-B$. Therefore, we can color the edges in the following way:
color-$1$ if  the value of the edge lays on the $x_1$-axis and it is different than $0$; color-$2$ if the value of the edge lays on the $x_2$-axis and it is different than $0$; $\ldots$ color-$n$ if the value of the edge lays on the $x_n$-axis and it is different than $0$; and white if the value of the edge is $0$. Considering any $2^m$ adjacent squares sharing a common vertex on the torus, we can see that all the squares can have only color-$1$ and white edges, or color-$2$ and white,$\ldots$ or color-$n$ and white or all white edges. We will be calling them color-$1$, color-$2$,$\ldots$ color-$n$ and white squares, respectively. Note that, by construction, no square can have only one non-white color. Also, the common edge between squares of different color is white.

We divide the unit square into color-$1$, color-$2$, $\ldots$ color-$n$ and white components. By component, we mean a monochromatic collection of squares, maximal with respect to inclusion, such that when seen on the surface of the torus it is connected. 

Let $C$ be a component and $\sigma=\partial C$ the union of faces enclosing $C$. First, we prove that $\sigma$ is a union of $n-1$-dimensional spheres (we will call them $n-1$-spheres). The proof is by induction on the number $k$ of squares in $C$. If $k=1$, the component is made of just one square, so $\sigma$ is the $n-1$-sphere enclosing the square. Let $k\ge 2$, and suppose that any component having less than $k$ squares is enclosed by a union of $n-1$-spheres. Let $C_0$ be a square in $C$ and let $C_1$ be the union of the squares in $C$ different than $C_0$. Then $C=C_0\cup C_1$. Moreover, $C_1$ can be seen as a disjoint union of components, each of them having less than $k$ squares, hence enclosed by a union of $n-1$-spheres. Thus, $\partial C_1$ is a union of $n-1$-spheres. The intersection $C_0\cap C_1$ can be either $0$-dimensional, $1$-dimensional, $\ldots$ or an $n-1$-dimensional object. In each of the cases, we obtain that $\sigma$ is the union of $n-1$-dimensional spheres. As any union of $n-1$-spheres can be seen as the union of injective images of $n-1$-spheres into the torus, we conclude that $\sigma$ is the union of injective images of $n-1$-spheres. Note that $\sigma$ is also the topological boundary of $C$. We will refer to it as a boundary of $C$.

Given a curve, we call \emph{gain} the sum of the values associated by $\psi$ to the edges forming it. A gain of value $r\cdot e_k$, with $r\in \mathbb{Z}^*$, can only be obtained through squares of color-$k$. Therefore, because the gain of each curve $[[0,e_k]]$ is $-e_k$, for $1\leqslant k \leqslant n$ the coloring must contain components of every non-white color. The boundary of a component is a union of injective images of $n-1$-spheres with white edges only. Whence, from $(*)$ above, the closed curves on the boundary of a component are contractible on the torus. Indeed $\overline{\psi_0}$ on such closed curves is zero. Here, by curve being contractible on the torus we mean that its homotopy class is $0$.

Let us consider any color-$1$ component. If a component is contractible on the torus any line that crosses the component will have the $e_1$ gain equal to 0 inside the component. The horizontal gain $-e_1$ is obtained only through color-$1$ squares, so there exists a color-$1$ component that is non-contractible on the torus. On the other hand, it follows from the following lemma and its corollary, that if a component has a boundary that consists of injective images of $n-1$-spheres whose contribution to the fundamental group of the torus is trivial, then it is either contractible on the torus or it contains loops generating the fundamental group of the torus. Whence, there exist color-$1$ component inside which we can obtain every gain, $-e_1,\ldots, -e_n$. This is a contradiction.

\end{proof}

\begin{lemma} Let $C$ be a component, and $i:C\rightarrow T$ be the inclusion map. Assume that  the boundary of $C$ is given by the image of an injective map from the $n-1$-sphere into $T$. Then
\[
i_*(\Pi_1(C))=\left\{
\begin{array}{l}
0 \\
\mathbb{Z}^n 
\end{array}\right.
\] 
\end{lemma}

\begin{proof} From the assumptions, let $\partial C$ be the image of the injective map $s: S^{n-1}\rightarrow T$, where $S^{n-1}$ denotes the $n-1$-sphere. Fix a point $x_0\in\partial C$. Let $f:S^1\rightarrow \partial C$ be a close curve, i.e. a map from the $1$-sphere, in the boundary of $C$. We denote by $p:\mathbb{R}^n\rightarrow T=\mathbb{R}^n/\mathbb{Z}^n$ the universal cover of $T$. We already establish that the edges on $\partial C$ are white, thus from $(*)$ we obtain that $f$ is null-homotopic. Therefor, if $j:\partial C\rightarrow T$ is the inclusion map, $(j\circ s)_*(\pi_1(S^{n-1}))=j_*(\pi_1(\partial C))=0\subset0=p_*(\pi_1(\mathbb{R}^2))$, hence for each lift $\widetilde{x_0}$ of $x_0$ to $\mathbb{R}^n$, there is a unique lift $\widetilde{s}:S^{n-1}\rightarrow\mathbb{R}^n$ lifting $s$ and the image containing $\widetilde{x_0}$. As $p$ is a covering map, it is a local homeomorphism of $\mathbb{R}^n$ with $T$, so $s$ being injective implies that the lifts are injective as well. Hence by Jordan-Brouwer Separation Theorem, $\widetilde{s}$ separates $\mathbb{R}^n$ into two open, path-connected components, of which the image of $\widetilde{s}$ is the common boundary. 

We fix a lift $\widetilde{x_0}$ and consider the lifting $\widetilde{s}:S^{n-1}\rightarrow\mathbb{R}^n$ of $s$ with $\widetilde{x_0}\in\widetilde{s}(S^{n-1})$. Let us denote by $\widetilde{U}$ the interior region defined by $\widetilde{s}$ and by $\widetilde{D}=\overline{\widetilde{U}}=\widetilde{U}\cup\textrm{Im}(\widetilde{s})$ the closure. We will prove that $p|_{\widetilde{D}}$ is injective. Let us assume the contrary, so there exist $x, y\in \widetilde{D}$ such that $x\ne y$ and $p(x)=p(y)$. Hence, there exists $a\in\mathbb{Z}^n_{\ne 0}$ such that $y=x+a$. On the other hand, $\widetilde{D}$ is the closure of a connected set, so it is connected and since it is locally path-connected, $\widetilde{D}$ is path-connected. Thus there exists a path $\widetilde{\gamma}:I\rightarrow \widetilde{D}$, where $I=[0,1]$,  with $\widetilde{\gamma}(0)=x$ and $\widetilde{\gamma}(1)=\widetilde{x_0}$. We denote by $\widetilde{\delta}$ the map from the wedge sum, $I\lor S^{n-1}$ of $I$ and $S^{n-1}$ along $z=s^{-1}(x_0)$:  
\begin{equation*} 
 \widetilde{\delta}(t)= \left\{ 
\begin{array}{rl} 
\widetilde{\gamma}(t) & \text{if } t\in I \\ 
\widetilde{s}(t)& \text{if } t\in S^{n-1}\\ 
\end{array} \right. 
\end{equation*} 

\noindent  By assumption, $x$ and $y$ are two different lifts of $p(x)$, so we can consider the lift $\widetilde{\delta}$ of $\delta=p\widetilde{\delta}$ with $\widetilde{\delta}(0)=x$ and the lift $\widetilde{\delta}'$ of $\delta$ with $\widetilde{\delta}'(0)=y$. We consider the map $\widetilde{\delta}+a:I\lor S^{n-1}\rightarrow\mathbb{R}^n$. We have $p(\widetilde{\delta}+a)=\delta$ and $(\widetilde{\delta}+a)(0)=y$, so by the unique lifting property $\widetilde{\delta}'=\widetilde{\delta}+a$. Now, every lift of $\delta$ contains a lift of $s$, since $\delta(t)=s(t)$, for $t\in S^{n-1}$. Whence $\widetilde{s}(t)=\widetilde{\delta}(t)$, for $t\in S^{n-1}$, is the lift of $s$ containing $\widetilde{x_0}$ and $\widetilde{s}'(t)=\widetilde{\delta}'(t)$, for $t\in S^{n-1}$ is the lift of $s$ containing $\widetilde{s}'(z)=\widetilde{s}(z)+a=\widetilde{x_0}+a$. Furthermore, $\widetilde{s}'=\widetilde{s}+a$. Thus $\textrm{Im}(\widetilde{s})\cap \textrm{Im}(\widetilde{s'})=\emptyset$. For if $\widetilde{s}(t_1)=\widetilde{s'}(t_2)$, for some $t_1,t_2\in S^{n-1}$, then $\widetilde{s}(t_1)=\widetilde{s}(t_2)+a$, hence $t_1\ne t_2$. Moreover, $s(t_1)=p\widetilde{s}(t_1)=p\widetilde{s'}(t_2)=s(t_2)$, and since $s$ is injective $t_1=t_2$. Whence $s(t_1)+a=s'(t_1)=s'(t_2)=s(t_1)$, contradicting $a\ne 0$. We obtained that the images of the liftings $\widetilde{s}$ and $\widetilde{s}'=\widetilde{s}+a$ of $s$ are disjoint. We consider the intersection $\widetilde{D}\cap\widetilde{D'}$, where $\widetilde{D'}$ is the closure of the interior region defined by $\widetilde{s'}$. We have $\widetilde{D}'=\widetilde{D}+a$, so $\mu(\widetilde{D}')=\mu(\widetilde{D})$, where by $\mu$ we denote the Lebesgue measure. Since $\textrm{Im}(\widetilde{s})$ and $\textrm{Im}(\widetilde{s'})$ are disjoint, we have $\textrm{Im}(\widetilde{s})\subset\widetilde{U'}$ or $\textrm{Im}(\widetilde{s})\subset(\widetilde{D'})^C$. If $\textrm{Im}(\widetilde{s})\subset\widetilde{U'}$, then $\widetilde{D}\subset \widetilde{U'}\subset \widetilde{D'}$ and $\mu(\widetilde{D'}\setminus \widetilde{D})=\mu(\widetilde{D})-\mu(\widetilde{D'})=0$. Having that  $U'\setminus D\subset D'\setminus D$, we obtain $\mu(U'\setminus D)=0$. This is a contradiction, for $U'\setminus D$ is open in $\mathbb{R}^n$, whence $\mu(U'\setminus D)>0$. We obtain $\textrm{Im}(\widetilde{s})\subset(\widetilde{D'})^C$. Similarly, $\textrm{Im}(\widetilde{s'})\subset(\widetilde{D})^C$. Hence, $\widetilde{D}\cap\widetilde{D'}=\emptyset$. On the other hand, by assumption, $y\in\widetilde{D}$ and $y=x+a\in\widetilde{D}+a=\widetilde{D'}$. This is a contradiction, so $p|_{\widetilde{D}}$ is injective.

Next, we need to prove that if $D=p(\widetilde{D})$, then $C=D$ or $C=\overline{D^C}$. First, we prove that $\textrm{Int}(C)=C\setminus\textrm{Im}(s)$ and $\textrm{Int}(\overline{C^C})=\overline{C^C}\setminus\textrm{Im}(s)=C^C$ are connected sets in $T$. This is to say that $s$ divides $T$ into two connected components: $T\setminus\textrm{Im}(s)=\textrm{Int}(C)\cup \textrm{Int}(\overline{C^C})$. As this components are open in $\mathbb{R}^n$, they are locally path-connected, and thus path-connected. 

We consider $C\setminus\textrm{Im}(s)$. Proof is by induction on number $n$ of squares in $C$. If $n=1$, the component $C$ is a square and the square without its border is connected. Let $n\ge2$, and suppose that the statement is true if $C$ consists of less than $n$ squares. Let $C_0$ be a square in $C$ with a face on the boundary $\textrm{Im}(s)$ and let $C_1$ be the union of squares in $C$ different than $C_0$. Since $s$ is an injective map from $S^{n-1}$ to $T$, the intersection $C_0\cap C_1$ is a connected collection of faces of $C_0$. Thus, the boundary $\partial C_1$ will be still homeomorphic to $S^{n-1}$. Thus, by induction hypothesis, $C_1\setminus\partial C_1$ is connected. Since the intersection $C_0\cap C_1$ is not contained in $\textrm{Im}(s)$, we obtain that $C\setminus\textrm{Im}(s)=(C_0\cup C_1)\setminus\textrm{Im}(s)$ is connected. A similar argument holds for $C^C$. 

Now, $p|_{\widetilde{D}}$ is injective, so $p(\widetilde{U})\cap p(\textrm{Im}(\widetilde{s}))=p(\widetilde{U})\cap\textrm{Im}(s)=\emptyset$, since $\widetilde{U}=\textrm{Int}(\widetilde{D})$. We have $p(\widetilde{U})\subset T\setminus\textrm{Im}(s)$. On the other hand, the interior $\widetilde{U}$ is connected. The covering map $p$ is continuous, whence $p(\widetilde{U})$ is connected. We obtain that $p(\widetilde{U})\subset\textrm{Int}(C)$ or $p(\widetilde{U})\subset \textrm{Int}(\overline{C^C})$, or equivalently $D=p(\widetilde{D})\subset C$ or  $D=p(\widetilde{D})\subset \overline{C^C}$.

Let us prove that $D=C$ or $D=\overline{C^C}$, the latter being equivalent to $C=\overline{D^C}$. Let us assume that $D\subset C$. This means that $p(\widetilde{U})\subset\textrm{Int}(C)$. Fix $x\in p(\widetilde{U})$. There exists $\widetilde{x}\in \widetilde{U}$ such that $x=p(\widetilde{x})$. Let $y\in \textrm{Int}(C)$. Since $\textrm{Int}(C)$ is path-connected, there exist a path $g:I\rightarrow \textrm{Int}(C)\subset T$ such that $g(0)=x$ and $g(1)=y$. Let $\widetilde{g}:I\rightarrow\mathbb{R}^n$ be the unique path lifting $g$ and starting at $\widetilde{x}=\widetilde{g(0)}$. Then $\widetilde{y}=\widetilde{g(1)}$ is a lift of $y$, i.e. $p(\widetilde{y})=y$. Moreover, $\textrm{Im}(\widetilde{g})\cap \textrm{Im}(\widetilde{s})=\emptyset$. For, if $\widetilde{z}\in \textrm{Im}(\widetilde{g})\cap \textrm{Im}(\widetilde{s})$, then $p(\widetilde{z})\in \textrm{Im}(g)\cap \textrm{Im}(s)$, a contradiction, since $\textrm{Im}(g)\subset\textrm{Int}(C)$ and $\textrm{Int}(C)\cap\textrm{Im}(s)=\emptyset$. We obtain that $\textrm{Im}(\widetilde{g})\subset\widetilde{U}$, so $\widetilde{y}\in\widetilde{U}$ and $y\in p(\widetilde{U})$. Whence, $p(\widetilde{U})\subset\textrm{Int}(C)$ and $D=C$. Similarly, we conclude that if $p(\widetilde{U})\subset \textrm{Int}(\overline{C^C})$, then $C=\overline{D^C}$.

We are now in position of proving the statement. Let $h:I\rightarrow T$ be a closed curve in $D$. Then there is a unique lift $\widetilde{h}:I\rightarrow\mathbb{R}^n$ of $h$ such that $\widetilde{h}(0)\in\widetilde{D}$. Moreover, since $p|_{\widetilde{D}}$ is injective and $h(0)=h(1)$, we have $\widetilde{h}(0)=\widetilde{h}(1)$, so $\widetilde{h}$ is a closed curve in $\mathbb{R}^n$, hence $[\widetilde{h}]=0$. Having that $h=p(\widetilde{h})$, we obtain $[h]=[p\circ \widetilde{h}]=p_*[\widetilde{h}]=0$. This means that if $\imath:D\rightarrow T$ denotes the inclusion map, then $\imath_*(\pi_1(D))=0$. But, we already proved that $C=D$ or $C=\overline{D^C}$. Thus, if $C=D$, then $i_*(\pi_1(C))=0$. On the other hand, if $C=\overline{D^C}$ , then as $\imath_*(\pi_1(D))=0$, Van Kampen's theorem implies  $i_*(\pi_1(C))=\mathbb{Z}^n$.

\end{proof}

\begin{corollary} Let $C$ be a component. Then
\[
i_*(\pi_1(C))=\left\{
\begin{array}{l}
0 \\
\mathbb{Z}^n
\end{array}\right.
\] 

\end{corollary}

\begin{proof} First, we consider the case when the $\textrm{Int}(C)$ is connected. In this case, as $\textrm{Int}(C)$ is locally path-connected, whence $\textrm{Int}(C)$ is path-connected. We already noticed that the boundary of $C$ is a union of $n-1$-spheres which whose fundamental group are trivial on the torus. Let $s:S^{n-1}\rightarrow T$ be an injective map into a part of the boundary of $C$. Arguing as in the previous lemma, every lift $\widetilde{s}$ of $s$ is an injective map on $\mathbb{R}^n$ and, by Jordan-Brouwer Separation Theorem, it separates $\mathbb{R}^n$ into two open, path-connected components, of which the image of $\widetilde{s}$ is the common boundary. Let us fix a lift $\widetilde{s}$ of $s$ and let $\widetilde{U}$ be the interior region defined by $\widetilde{s}$. By the previous lemma, $p|_{\widetilde{D}}$ is injective, where $\widetilde{D}=\overline{\widetilde{U}}$. Two possibilities may occur: $p(\widetilde{U})\cap \textrm{Int}(C)\ne\emptyset$ or $p(\widetilde{U})\cap \textrm{Int}(C)=\emptyset$. Let us consider the case $p(\widetilde{U})\cap \textrm{Int}(C)\ne\emptyset$. As $\textrm{Int}(C)$ is path-connected, we obtain $\textrm{Int}(C)\subset p(\widetilde{U})$. Thus, we have $\textrm{Int}(C)\subset p(\widetilde{U})$ or $p(\widetilde{U})\cap \textrm{Int}(C)=\emptyset$, the latter being equivalent to $p(\widetilde{U})\subset C^C$. If $\textrm{Int}(C)\subset p(\widetilde{U})$, then $\textrm{Int}(C)\subset (\bigcup_{a\in\mathbb{Z}^n}p(\widetilde{U}+a))=p(\bigcup_{a\in\mathbb{Z}^n}(\widetilde{U}+a))$, where $\widetilde{U}+a$, when $a$ ranges through $\mathbb{Z}^n$, represents the interior regions of all liftings of $s$. Hence, $p^{-1}(\textrm{Int}(C))\subset\bigcup_{a\in\mathbb{Z}^n}(\widetilde{U}+a)$. On the other hand, if $p(\widetilde{U})\cap \textrm{Int}(C)=\emptyset$, then $p(\bigcup_{a\in\mathbb{Z}^n}(\widetilde{U}+a))\cap \textrm{Int}(C)=\emptyset$, so $p^{-1}(\textrm{Int}(C))\subset(\bigcup_{a\in\mathbb{Z}^n}(\widetilde{U}+a))^C=\bigcap_{a\in\mathbb{Z}^n}(\widetilde{U}+a)^C$. We conclude that $p^{-1}(C)\subset\bigcup_{a\in\mathbb{Z}^n}(\widetilde{D}+a)$ or $p^{-1}(C)\subset\bigcap_{a\in\mathbb{Z}^n}(\widetilde{U}+a)^C$. 

Now, each of the $n-1$-spheres making the boundary of $C$ is lifted to $n-1$-spheres through $p$ and each lift defines an interior and an exterior region. Let $D$ be the union of the closures of the interior regions and $E$ be the intersection of the closures of the exterior regions. By the previous argument, we obtain that $p^{-1}(C)\subset D$ or $p^{-1}(C)\subset E$, and since $p$ is surjective, we have $C\subset p(D)$ or $C\subset p(E)$. 

If $C\subset p(E)$, we actually have the equality $C=p(E)$. Indeed, if $p(E)\setminus C\ne\emptyset$, then exists a square $S$ in $p(E)$ not belonging to $C$ such that $S\cap C\ne\emptyset$. For, if that is not the case, $p(E)=(\bigcup_{S\in p(E)} S)\cup C$ would be a disconnection, which would contradict the fact that $p(E)$ is connected. But this would mean that there exists $x\in S$ such that $p^{-1}(x)\subset \bigcup_{a\in\mathbb{Z}^2}(\widetilde{U}+a)$, where $\widetilde{U}$ is an interior region of a lifting of one of the $n-1$-spheres making the border of $C$. A contradiction, since there exists $y\in p^{-1}(x)$ such that $y\in E$ and $E\cap \bigcup_{a\in\mathbb{Z}^n}(\widetilde{U}+a)=\emptyset$.

On the other hand, by the previous lemma and Van Kampen's theorem, we have $j_*(\pi_1(p(D)))=0$, where $j:p(D)\rightarrow T$ is the inclusion map. Having that $D\cup E=\mathbb{R}^n$, we obtain that $p(D)\cup p(E)=T$ and, by Van Kampen's theorem, $k_*(\pi_1(p(E)))=\mathbb{Z}^n$, where $k:p(E)\rightarrow T$ is the inclusion map. By the previous discussion, $C\subset p(D)$ or $C=p(E)$, which ends the proof in the case when $\textrm{Int}(C)$ is connected.

Any component $C$ can be seen as the union of components with connected interiors with intersections of dimension strictly less than $n-1$. The statement now follows from Van Kampen's theorem.
  
\end{proof}

The arguments behind the proof of Theorem \ref{thm} are of a geometric nature. They can be modified to apply to any linearly independent set of $n$ lines to obtain the following generalization.

\begin{corollary}
Let $K$ be a compact set such that $\mathbb{R}^n=K+\mathbb{Z}^n$ and let $l_1,\ldots, l_n\subseteq\mathbb{R}^n$ be $n$-linearly independent one dimensional vector subspaces. Then  $(K-K)\cap\mathbb{Z}^n\not\subseteq l_1\cup l_2\cup\ldots\cup l_n$.
\end{corollary}

\end{document}